\documentclass[11pt]{amsart}
\usepackage{amsfonts,amssymb,amsthm,amsmath,amsxtra,amscd,verbatim,eucal}
\usepackage[all]{xy}
\usepackage[dvips]{graphics}

\theoremstyle{plain}
\newtheorem{thm}{Theorem}[section]
\newtheorem{lem}[thm]{Lemma}
\newtheorem{prop}[thm]{Proposition}
\newtheorem{cor}[thm]{Corollary}

\theoremstyle{definition}
\newtheorem{defi}[thm]{Definition}

\theoremstyle{remark}
\newtheorem{eg}[thm]{Example}

\newtheorem{rmk}[thm]{Remark}

\def\C{{\mathbb C}}

\def\R{{\mathbb R}}
\def\Q{{\mathbb Q}}
\def\P{{\mathbb P}}

\def\AA{\mathcal{A}}

\def\O{\mathcal{O}}
\def\LL{\mathcal{L}}
\def\M{\mathcal{M}}

\def\J{\mathcal{J}}
\def\I{\mathcal{I}}

\def\d{\delta}

\def\e{\eta}

\def\m{\mu}
\def\om{\omega}

\def\x{\xi}

\def\L{\Lambda}
\def\S{\Sigma}
\def\Om{\Omega}

\def\o{\circ}
\def\ov{\overline}

\def\lrd{\lfloor}
\def\rrd{\rfloor}
\def\lru{\lceil}
\def\rru{\rceil}

\def\.{\cdot}
\def\({\Big{(}}
\def\){\Big{)}}
\def\^{\widehat}
\def\~{\widetilde}

\renewcommand{\and}{ \quad \text{and} \quad }
\renewcommand{\for}{ \quad \text{for} \ \, }

\DeclareMathOperator{\codim} {codim}

\DeclareMathOperator{\rk} {rk}

\DeclareMathOperator{\Hom}  {Hom}

\DeclareMathOperator{\Proj} {Proj}

\DeclareMathOperator{\reg} {reg}

\DeclareMathOperator{\Sing} {Sing}

\DeclareMathOperator{\Bl} {Bl}
\DeclareMathOperator{\val} {val}

\DeclareMathOperator{\mult} {mult}
\DeclareMathOperator{\ord} {ord}

\DeclareMathOperator{\Span} {Span}

\DeclareMathOperator{\Jac} {Jac}

\DeclareMathOperator{\mld} {mld}

\title{A vanishing theorem for log canonical pairs}

\author{Tommaso de Fernex}
\address{Department of Mathematics, University of Utah, 155 South 1400 East,
Salt Lake City, UT 48112-0090, USA}
\email{defernex@math.utah.edu}

\author{Lawrence Ein}
\address{Department of Mathematics,
University of Illinois at Chicago,
851 South Morgan Street,
Chicago, IL 60607-7045, USA}
\email{ein@math.uic.edu}

\thanks{The first author was partially supported by
NSF under Grant DMS-0548325.
The second author was partially supported by NSF
under Grant DMS-0700774}

\keywords{Vanishing theorem, log canonical
singularities, inversion of adjunction, Castelnuovo--Mumford
regularity.}

\begin{document}

\baselineskip 13pt

\begin{abstract}
Using inversion of adjunction, we deduce from Nadel's theorem
a vanishing property for ideals sheaves on
projective varieties, a special case of which recovers
a result due to Bertram--Ein--Lazarsfeld.
This enables us to generalize to a large class of projective schemes
certain bounds on Castelnuovo--Mumford regularity previously obtained
by Bertram--Ein--Lazarsfeld in the smooth case and by Chardin--Ulrich for
locally complete intersection varieties with rational singularities.
Our results are tested on several examples.
\end{abstract}

\maketitle

\section{Introduction}

The following vanishing property and its implications to questions related to
Castelnuovo--Mumford regularity are the main goal of the paper.

\begin{thm}\label{thm:vanishing}
Let $X$ be a locally complete intersection projective variety
with rational singularities, and let
$V \subset X$ be a pure-dimensional proper subscheme with no embedded components.
Suppose that $V$ is scheme-theoretically given by
$$
V = H_1 \cap \dots \cap H_t
$$
for some divisors $H_i \in |\LL^{\otimes d_i}|$, where $\LL$ is a
globally generated line bundle on $X$ and $d_1 \ge \dots \ge d_t$.
If the pair $(X,eV)$ is log canonical where $e = \codim_XV$,
and $\AA$ is an ample line bundle on $X$, then
$$
H^i(X,\om_X\otimes \LL^{\otimes k}\otimes\AA\otimes \I_V) = 0
\for i > 0, \ k \ge d_1 + \dots + d_e.
$$
\end{thm}

In the special case when both $X$ and $V$ are assumed to be smooth
this result was proven in \cite{BEL}
using vanishing theorems on certain residual intersections on the
blow up of $X$ along $V$.

The proof of Theorem~\ref{thm:vanishing} follows a novelle
approach, based on a result on inversion of adjunction from \cite{EM1}.
The key idea is to produce a suitable $\Q$-scheme $Z$ on $X$
whose multiplier ideal $\J(X,Z)$ coincides with $\I_V$ in a
neighborhood of $V$. We do this by taking
$Z = (1 - \d)B + \d e V$, where $B$ is the base scheme of
$|\LL^{\otimes d_1 + \dots + d_e} \otimes \I_V^e|$ and
$0 < \d \ll 1$. We reduce in this way to
a setting in which one can directly apply Nadel's vanishing theorem on $X$,
without having to blow up.
Our argument, once restricted to the case of smooth varieties,
gives in particular a very simple proof of the main result in \cite{BEL}.
We remark that there are other $\Q$-schemes
$Z$ such that $\J(X,Z)$ coincides with $\I_V$ near $V$ (e.g., $Z = eV$),
but the vanishing properties one can deduce from them are in general
not as strong (cf. Remark~\ref{rmk:d_i-equal}).

Vanishing theorems of the above kind
are motivated by questions concerning Castelnuovo--Mumford regularity.
We recall that a proper closed subscheme $V$ of $\P^n$ is said to be {\it $m$-regular}
(in the sense of Castelnuovo--Mumford) if
\begin{equation}\label{eq:regularity}
H^i(X,\I_V(m-i)) = 0 \for i > 0.
\end{equation}
The {\it regularity} of $V$ is the minimum
value of $m$ for which \eqref{eq:regularity}
holds, and is denoted by $\reg(V)$.
The regularity of a scheme is related to the complexity of the
associated ideal sheaf (and its sygyzies)
and, apart of a few special cases, is in general difficult to compute.
For more on the notion of regularity, we refer to \cite[Section~1.8]{Laz-I}.

When $V$ is a smooth projective variety, bounds on regularity
were determined in terms of the degrees of $V$ or of its
defining equations. It is expected that $\reg(V) \le \deg V - e + 1$
for any non-degenerate variety $V \subset \P^n$ of codimension $e$.
The curve case, originally studied by Castelnuovo \cite{Cas}, was
completely settled (also for possibly singular curves) in \cite{GLP},
and the bound has been established for smooth surfaces in \cite{Pin,Laz}
and smooth threefolds in \cite{Ran}; see also \cite{Kwa}.

The results of \cite{BEL} showed that in general
potentially stronger bounds can be determined in terms of the
degrees of a set of equations defining $V$, rather than the degree of $V$.
It is convenient to introduce the notation
$$
d(V) := \min \{d_1 + \dots + d_e \mid
\text{$V$ is cut out by forms of degrees $d_1 \ge \dots \ge d_t$}\},
$$
where $e = \codim_{\P^n}(V)$. Then the result in \cite{BEL} says
that if $V \subset \P^n$ is a smooth variety of codimension $e$, then
$$
\reg(V) \le d(V) - e+1,
$$
and equality holds if and only if $V$ is a complete intersection in $\P^n$.

Here we address the same question allowing $V$ to a singular
variety or, more generally, a singular scheme in $\P^n$.
It was proven in \cite{CU} using liaison theory
that the above bound also holds
in the case $V$ is a locally complete intersection variety
with rational singularities.
By contrast, the regularity is in general sensitive
to the singularities, and there are in fact examples
of subschemes whose regularity grows doubly exponentially
in the generating degrees (cf. \cite{BS}).

In light of Theorem~\ref{thm:vanishing}, it becomes natural to
impose conditions within the context of singularities of pairs.
We first discuss the case in which $V$ is a variety,
where we can impose conditions directly on $V$.

If $V$ is a locally complete intersection variety, then we simply
assume that $V$ has log canonical singularities.
Since locally complete intersection varieties
with rational singularities have canonical (and hence log canonical) singularities,
this particular case of our result includes, essentially, the main result
in \cite{CU} (cf. Remarks~\ref{rmk:CU2} and~\ref{rmk:CU1}).

More generally, we consider varieties $V$ that are not necessarily
locally complete intersection.
We still require that $V$ is normal and $\Q$-Gorenstein.
In order to compare the singularities of $V$ with
those of the pair $(\P^n,eV)$, we rely on a more general
version of inversion of adjunction, proven in \cite{EM2,Kaw}.
For any positive integer $r$ such that $rK_V$ is Cartier,
there is a subscheme $\S_r \subset V$, only supported
on the locus where $V$ is not locally complete intersection,
which measures---so to speak---how far
$V$ is from being locally complete intersection
(see Definitions~\ref{def:lci-subscheme} and~\ref{def:lci-lc-sing} below);
one has that $V$ is locally complete intersection if and only if $\S_r = \emptyset$.
We therefore assume that the pair $(V,\frac 1r \S_r)$ has log canonical singularities.
This condition is independent of the choice of $r$;
if $V$ is locally complete intersection, then the condition is equivalent to $V$
being log canonical.

In the same spirit of \cite{BEL}, we obtain the following results.

\begin{cor}\label{thm:main-var-case}
Let $V \subset \P^n$ is a normal $\Q$-Gorenstein variety
of codimension $e$, and
suppose that the pair $(V,\frac 1r \S_r)$ has log canonical singularities,
where $\S_r$ is defined as above.
Then
$$
\reg(V) \le d(V) - e+1
$$
and equality holds if and only if $V$ is a complete intersection in $\P^n$.
\end{cor}

\begin{cor}\label{cor:proj-normal-CM}
With the same assumptions as in Corollary~\ref{thm:main-var-case},
if $d(V) \le n+1$, then
$V$ is projectively normal. Moreover, if $V$ has rational singularities
and $d(V) \le n$, then $V$ is projectively
Cohen--Macaulay.
\end{cor}

An important feature of our method
is that it allows us to consider a much larger class of subschemes of $\P^n$
by lifting our conditions on singularities to the pair $(\P^n,eV)$,
as the corresponding notions on $V$ would not make sense
in such generality. We obtain in this way
the following result (which, by inversion of adjunction,
includes Corollary~\ref{thm:main-var-case} as a special case).

\begin{cor}\label{thm:main-gen-case}
Let $V \subset \P^n$ be a pure-dimensional subscheme with no embedded components,
and assume that the pair $(\P^n,eV)$ is log canonical, where $e = \codim_{\P^n}V$.
Then
$$
\reg(V) \le d(V) - e+1
$$
and equality holds if and only if $V$ is a complete intersection in $\P^n$.
\end{cor}

The hypotheses of this result are satisfied, for instance, by examples
arising from classical constructions, such as determinantal varieties
and generic projections. It is interesting to observe that these constructions
in general produce varieties that are not locally complete intersection
and, at least in the second case, not normal.
Moreover, a construction via residual intersections provides a family of
examples whose regularity is precisely one less the one expected
for complete intersections (a characterization is given in
Proposition~\ref{prop:CU1} below).
These examples are discussed in the last section of the paper.

A different proof of \cite{BEL} was given in \cite{Ber}
by applying Nadel vanishing theorem to a certain multiplier ideal on
the blow up of $X$ along $V$. In fact,
the same approach can be extended, without
much change, to the case (considered in \cite{CU})
in which $V$ is a locally complete
intersection variety with rational singularities.

What inversion of adjunction gives us is a way to go beyond
these conditions on singularities, as it allows us to avoid
altogether to work on the blow up of $X$ along $V$.
Moreover, once the deep result from inversion of adjunction is granted,
the approach followed in this paper gives a quite simple and clean
proof which applies, uniformly, to all cases.

The authors are grateful to Rob Lazarsfeld for bringing \cite{Ber} to
their attention, and to Victor Lozovanu for pointing out 
that the line bundle $\AA$ in Theorem~\ref{thm:vanishing}
needs to be ample, instead of just nef and big.

\section{Vanishing theorems for multiplier ideal sheaves}

Let $X$ be a variety with rational singularities,
and assume that its canonical class $K_X$ is Cartier.
Let $\om_X \cong \O_X(K_X)$ be the dualizing sheaf.

If $\LL$ is a nef and big line bundle on $X$, then
$$
H^i(X,\om_X \otimes\LL) = 0 \for i > 0
$$
by the Kawamata--Viehweg vanishing theorem.
Multiplier ideals arise naturally in this context
as a way to encode this property from different birational models of $X$.

Consider a formal linear combination
$$
Z = \sum_{j=1}^k c_j Z_j,
$$
where $Z_j \subset X$ are proper closed subschemes
defined by ideal sheaves $\I_{Z_j} \subseteq \O_X$,
and $c_j \in \R$. The {\it multiplier ideal sheaf}
$\J(X,Z)$ of the pair $(X,Z)$ is then defined by
$$
\J(X,Z) = f_*\O_X\big(\big\lru K_{Y/X} - \sum c_j f^{-1}(Z_j)
\big\rru\big) \subseteq \O_X,
$$
where $f \colon Y \to X$ is any log resolution of $(X,Z)$,
$K_{Y/X}$ is the relative canonical divisor of $f$,
and $f^{-1}(Z_j)$ is the subscheme of $Y$ defined by $\I_{Z_j}\.\O_Y$.
We recall that the definition of log resolution
requires that $f$ is a proper birational morphism
from a smooth variety $Y$ such that each $f^{-1}(Z_j)$ and the
exceptional locus of $f$ are divisors whose supports form, together,
a simple normal crossings divisor.

By applying Kawamata--Viehweg vanishing theorem on log resolutions,
one obtains the following vanishing theorem.

\begin{thm}[Nadel vanishing theorem]
With the above notation, suppose that $A_j$ and $M$ are Cartier divisors on $X$
such that $\O_X(A_j) \otimes \I_{Z_j}$ is globally generated for each $j$
and $M-\sum c_jA_j$ is nef and big.
Then
$$
H^i(X,\om_X \otimes \O_X(M) \otimes \J(X,Z)) = 0 \for i > 0.
$$
\end{thm}

We refer to \cite{Laz-I} for proofs,
more general formulations of these vanishing theorems,
and more general settings where multiplier ideal sheaves are defined.

Broadly speaking, the main aim of this paper is to determine sufficient
conditions for the vanishing of the higher cohomology groups
$$
H^i(X,\om_X \otimes \M \otimes \I),
$$
where $\I \subset \O_X$ is an ideal sheaf and $\M$ is a line bundle on $X$.

As one may expect, the strategy is very simple, and relies on finding
formal linear combinations $Z = \sum c_jZ_j$ of proper closed subschemes
of $X$ such that $\I$ is realized as the multiplier ideal sheaf
$$
\I = \J(X,Z),
$$
so that conditions on $\M$ can be determined
to ensure to be in the setting of Nadel vanishing theorem.
One should notice however that different choices of $Z$ may lead to different
conditions on $\M$ (cf. Remark~\ref{rmk:d_i-equal} below).

A simple application of inversion of adjunction, which is the topic of
the next section, is the key ingredient that allows us to determine good
descriptions of certain ideal sheaves as multiplier ideal sheaves.

\section{Inversion of adjunction}

Let $X$ be a normal variety
such that $K_X$ is $\Q$-Cartier (we say that $X$ is
{\it $\Q$-Gorenstein}), and let $Z = \sum c_jZ_j$ be
a formal linear combination of proper closed subschemes of $X$,
with $c_j \in \R$.

For any closed subset $W \subseteq X$, we define the
{\it minimal log discrepancy}
$$
\mld(W;X,Z) := \inf_{f(E) \subseteq W}
\big\{\ord_E(K_{Y/X})+1 - \sum c_j\val_E(Z_j) \big\},
$$
where the infimum is computed by considering all log resolutions
$f \colon Y \to X$ of $(X,Z)$ and all prime divisors $E$ on any such $Y$
with center $f(E)$ contained in $W$.
If $W = \{p\}$, where $p \in X$ is a closed point, then we denote this number
by $\mld(p;X,Z)$.
If $W$ is irreducible and
$\e_W$ is the generic point of $W$, then we similarly define $\mld(\e_W;X,Z)$
by taking the infimum over all prime divisors $E$
with center $f(E) = W$.
For more on minimal log discrepancies, we refer to \cite{Ko0}.

It is easy to see that minimal log discrepancies can either
be $\ge 0$ or $-\infty$. The pair $(X,Z)$ is said to be {\it log canonical}
(or to have {\it log canonical singularities}) if $\mld(X;X,Z) \ge 0$.

We will use the following consequence of inversion of adjunction.

\begin{prop}\label{lem:mld1}
Let $X$ be a locally complete intersection normal variety, and
let $V \subset X$ be a proper subscheme of codimension $e$, scheme-theoretically
given by
$$
V = H_1 \cap \dots \cap H_t
$$
for some divisors $H_i \in |\LL^{\otimes d_i}|$, where $\LL$ is a
globally generated line bundle on $X$ and $d_1 \ge \dots \ge d_t$.
Suppose that both $X$ and $V$ are smooth at the generic point of each irreducible
component of maximal dimension of $V$.
Then for any point $p \in V$ there are sufficiently general
$D_i \in |\LL^{\otimes d_i}\otimes \I_V|$, for $i=1,\dots,e$, such that
$$
\mld(p;X,eV) = \mld(p;X,D_1 + \dots + D_e).
$$
\end{prop}

\begin{proof}
First we observe that $\LL^{\otimes d_1} \otimes \I_V$ is globally generated.
Since one can compute minimal log discrepancies by using a log resolution
where the inverse image of $p$ is a simple normal crossing divisor,
it follows in particular that if $D_1 \in |\LL^{\otimes d_1}\otimes\I_V|$ is
sufficiently general, then
\begin{equation}\label{eq:1}
\mld(p;X,V + Z) = \mld(p;X,D_1 + Z)
\end{equation}
for any formal linear combination $Z = \sum c_jZ_j$ of proper subschemes of $X$.

We proceed by induction on $e \ge 1$.
The above formula with $Z = 0$ verifies the statement for $e=1$,
so we can assume that $e \ge 2$.
We fix general $D_i \in |\LL^{\otimes d_i}\otimes \I_V|$
in the following order: we first choose a general $D_1$, then a general $D_2$
(the generality condition on $D_2$ may depend on the choice of $D_1$), and so forth.
Note that
$$
V = D_1 \cap \dots \cap D_e \cap H_{e+1} \cap \dots \cap H_t.
$$

Note that $D_1$ is a locally complete intersection scheme,
and it is regular in codimension one, and hence normal,
if chosen with sufficient generality.
Indeed, since the base locus of $|\LL^{\otimes d_1}\otimes\I_V|$ is
equal to the support of $V$,
the restriction of a general $D_1$ to the open set $X^\o = X \setminus V$
is smooth on the regular locus of $X^\o$ and
intersects properly the singular locus of $X^\o$, and thus
it is regular in codimension one.
On the other hand, since both $X$ and $V$ are smooth at the generic point
of each irreducible component of $V$ of codimension $e$,
a general choice of $D_1$ will also be smooth at that point.
Recalling that we are assuming that $e \ge 2$, we conclude that
a sufficiently general $D_1$ is regular in codimension one.

Since $D_1$ is smooth at the generic point of each irreducible
component of $V$ of maximal dimension,
we can apply induction by restricting to $D_1$. By \eqref{eq:1} with $Z = (e-1)V$,
\cite[Corollary~8.2]{EM2} applied twice, and induction,
we get
\begin{align*}
\mld(p;X,eV)
&= \mld(p;X,D_1+(e-1)V) \\
&= \mld(p;D_1,(e-1)V) \\
&= \mld(p;D_1,D_2|_{D_1} + \dots + D_e|_{D_1}) \\
&= \mld(p;X,D_1 + \dots + D_e).
\end{align*}
\end{proof}

Suppose now that $V$
is a normal $\Q$-Gorenstein variety, and let $r$ be a positive integer
such that $rK_V$ is Cartier. The image of the canonical map
$(\Om_V^{\dim V})^{\otimes r} \to \O_V(rK_V)$ is equal to
$I_r \. \O_V(rK_V)$ for some ideal sheaf $I_r \subseteq \O_V$.
Let
$$
J_r := (\overline{\Jac_V^r} : I_r) \subseteq \O_V,
$$
where $\Jac_V\subseteq \O_V$ is the Jacobian ideal sheaf of $V$.
By \cite[Corollary~9.4]{EM2}, the ideal sheaves $\Jac_V^r$ and $I_r\.J_r$
have the same integral closure.
Let $\S_r \subset V$ be the subscheme defined by $J_r$.

\begin{defi}\label{def:lci-subscheme}
With the above notation,
we call $\S_r$ the {\it $r$-th lci-defect subscheme} of $V$.
\end{defi}

By taking an embedding of $V$ in a smooth ambient variety $X$,
it follows from \cite[Theorem~8.1]{EM2} that if $s$
is another positive integer such that $sK_V$ is Cartier, then
$$
\mld(W;V,\tfrac 1r \S_r) = \mld(W;X,eV)
= \mld(W;V,\tfrac 1s \S_s)
$$
for every closed subset $W \subseteq V$, where $e = \codim_XV$
(see also \cite{Kaw}). In particular,
the property that $(V,\tfrac 1r \S_r)$ be log canonical
(in some neighborhood or on the whole $V$) is independent
of the choice of $r$. As we will be using this condition on singularities,
we fix the following terminology.

\begin{defi}\label{def:lci-lc-sing}
A normal $\Q$-Gorenstein variety $V$ is said to be
{\it lci-defectively log canonical},
or to have {\it lci-defectively log canonical singularities},
if for some (equivalently, for every) positive integer $r$
such that $rK_X$ is Cartier, the pair $(V,\frac 1r \S_r)$
is log canonical.
\end{defi}

\begin{rmk}
Being lci-defectively log canonical is in general a
stronger condition than being log canonical; the two
conditions are equivalent if $V$ is
a locally complete intersection variety.
\end{rmk}

One immediately obtains from the proposition the
following variant of inversion of adjunction.

\begin{cor}\label{thm:mld1}
Let $X$ be a smooth variety, and
let $V \subset X$ be proper subvariety of codimension $e$, scheme-theoretically
given by
$$
V = H_1 \cap \dots \cap H_t
$$
for some divisors $H_i \in |\LL^{\otimes d_i}|$, where $\LL$ is a line bundle
on $X$ and $d_1 \ge \dots \ge d_t$.
Assume that $V$ is a normal variety with $\Q$-Gorenstein singularities,
let $r$ be a positive integer such that $rK_V$ is Cartier,
and let $\S_r$ be the $r$-th lci-defect subscheme of $V$.
Then for every point $p \in V$,
there are sufficiently general $D_i \in |\LL^{\otimes d_i}\otimes \I_V|$,
for $i=1, \dots, e$, such that
$$
\mld(p;V,\tfrac 1r \S_r) = \mld(p;X,D_1 + \dots + D_e).
$$
In particular, $(X,D_1 + \dots + D_e)$ is log canonical
in a neighborhood of $p$ if and only if
$V$ is lci-defectively log canonical in a neighborhood of $p$.
\end{cor}

\begin{proof}
By \cite[Theorem~8.1]{EM2} we have
$$
\mld(p;V,\tfrac 1r \S_r) = \mld(p;X,eV),
$$
and thus the first assertion follows from Proposition~\ref{lem:mld1}.
The last assertion is a consequence of the general fact that
if $Y$ is a normal $\Q$-Gorenstein variety, $Z \subset Y$
is a proper closed subscheme, and $c > 0$, then
for every $q \in Y$ the pair
$(Y,cZ)$ is log canonical near $q$ if and only
if $\mld(q;Y,cZ) \ge 0$.
\end{proof}

\section{Proof of Theorem~\ref{thm:vanishing}}

Let $X$ be a locally complete intersection
projective variety with rational singularities,
and consider on $X$ a globally generated line bundle $\LL$ and an ample line bundle $\AA$.
Let $V \subset X$ be a pure-dimensional
proper subscheme with no embedded points, scheme-theoretically given by
$$
V = H_1 \cap \dots \cap H_t
$$
for some divisors $H_i \in |\LL^{\otimes d_i}|$, where $d_1 \ge \dots \ge d_t$.
Let $e = \codim_XV$, and suppose that the pair $(X,eV)$ is log canonical.

We consider the base scheme $B \subset X$ of the linear system
$$
|\LL^{\otimes (d_1 + \dots + d_e)}\otimes\I_V^e|.
$$
Given any $p \in V$, we
fix sufficiently general divisors
$D_i \in |\LL^{\otimes d_i}\otimes\I_V|$ for $i=1,\dots,e$.
Note that
$$
D = D_1 + \dots + D_e \in |\LL^{\otimes (d_1 + \dots + d_e)}\otimes\I_V^e|.
$$

The following property is of independent interest.

\begin{prop}\label{prop:X-V-smooth}
Let $X$ be a locally complete intersection normal variety.
Let $V \subset X$ be a closed subscheme,
and let $\e_W$ be the generic point of
an irreducible component $W$ of $V$.
Suppose that $(X,eV)$ is log canonical at $\e_W$,
where $e$ is the codimension of $W$ in $X$.
Then both $X$ and $V$ are smooth at $\e_W$.
\end{prop}

\begin{proof}
This can easily be seen using inversion of adjunction, as follows.
We work locally near $\e_W$, so that all minimal log discrepancies over $\e_W$
considered in the proof can be thought as minimal log discrepancies
over $W$; this allows us to apply inversion of adjunction where needed.

We also assume that $X$ is a complete intersection subvariety of a smooth
variety $M$. Let $c$ be the codimension of $X$ in $M$.
If $c = 0$, then $X$ is smooth at $\e_W$. Suppose otherwise that $c > 0$, and let
$H, L \subset M$ be general (sufficiently positive) hyperplane
sections vanishing, respectively, along $X$ and $V$.
It follows by inversion of adjunction that
$$
\mld(\e_W;M,cH + eL) = \mld(\e_W;M,cX+eV) = \mld(\e_W;X,eV) \ge 0
$$
On the other hand, since $M$ is smooth and $\codim_MW = c+e$,
we have $\mld(\e_W;M,cH + eL) \le 0$, with equality holding only
if both $H$ and $L$ are smooth at $\e_W$. We conclude that this is the case,
and in particular that $\mld(\e_W;X,eV) = 0$.
As we can assume that the equation of $H$ belongs to a regular sequence on $M$
locally cutting out
$X$ as a complete intersection near $\e_W$, it follows by induction on $c$
that $X$ is smooth at $\e_W$.

Let $S$ be, locally near $\e_W$, a general hyperplane section of $X$
vanishing along $V$ (equivalently, let $S = L \cap X$). Then
$\mld(\e_W;X,eS) = \mld(\e_W;X,eV) = 0$,
which implies that $S$ is smooth at $\e_W$.
Applying inversion of adjunction on $X$, we obtain
$$
\mld(\e_W;S,(e-1)V) = \mld(\e_W;X, S + (e-1)V) = \mld(\e_W;X,eV) = 0.
$$
We conclude by induction on $e$ that $V$ is smooth at $\e_W$.
\end{proof}

\begin{rmk}
This property is related to a conjecture of Shokurov, which says that if
$(X,B)$ is an effective log pair (namely, $X$ is normal and $B$ is
an effective $\Q$-divisor such that $K_X + B$ is $\Q$-Cartier)
and $\x \in X$ is a Grothendieck point of codimension $e$, then
$\mld(\x;X,B) \le e$, and moreover
$\mld(\x;X,B) \le e - 1$ if $X$ is singular at $\x$ (cf. \cite{Amb}).
Similar arguments as in the proof of Proposition~\ref{prop:X-V-smooth}
show that this conjecture is true if $X$ is a locally complete intersection
variety. Indeed, it suffices to show that the condition
$\mld(\x;X,B) > e-1$ implies that $X$ is smooth at $\x$.
Assuming that $X$ is locally complete intersection,
one takes, locally near $\x$,
an embedding of $X$ as a complete intersection subvariety of a smooth
variety $M$. Then one can reduce the codimension $c$ of the embedding
(assuming that $c > 0$) as follows. First notice that
$B$ is $\Q$-Cartier, and thus $B = rA$, where $A$ is an effective Cartier divisor
and $r \in \Q_+$.
We regard $A$ as a subscheme of $X$, and hence of $M$.
Working locally near $\x$ and taking a  general (sufficiently positive) hyperplane
section $H \subset M$ vanishing along $X$, inversion of adjunction gives
$$
\mld(\x;M,cH + rA) = \mld(\x;M,cX+rA) = \mld(\x;X,rA) > e-1.
$$
This implies that $H$ is smooth at $\x$, and thus we conclude
by induction on $c$ that $X$ is smooth too at $\x$, as required.
\end{rmk}

We now come back to the setting introduced at the beginning of the section.
By Proposition~\ref{prop:X-V-smooth}, we see that
both $X$ and $V$ are smooth at the generic point of each irreducible
component $V_i$ of $V$. We can therefore apply
Proposition~\ref{lem:mld1}, which implies that $(X,D)$ is log canonical near $p$.
Observing that $\O_X(-D) \subseteq \I_B$, we conclude that
$(X,B)$ is log canonical near $p$, and therefore in a whole
neighborhood of $V$.

Note that
$$
\val_E(B) \ge e\.\val_E(V)
$$
for every prime divisor $E$ over $X$, and
this is an equality if $E = E_i$ is the divisor dominating an
irreducible component $V_i$ of $V$ that is extracted by the blow up
$$
\m\colon X' = \Bl_VX \to X
$$
of $X$ along $V$.
The inequality follows by the inclusion $\I_B \subseteq \I_V^e$.
For the second assertion, just notice that
if $D_1,\dots,D_e$ are generally chosen
as above, then their equations locally cut out $V$
at the generic point of $V_i$,
and hence $D$ has multiplicity $e$ at that point.

\begin{lem}
With the above notation, if $0 < \d \ll 1$, then
$$
\J(X,(1-\d)B + \d e V) = \I_{V \cup W},
$$
where $W$ is a closed subscheme of $X$ disjoint from $V$.
\end{lem}

\begin{proof}
Observe that if $E$ is a prime divisor on some model $Y$ over $X$,
with center contained in $V$, then
$$
\val_E ((1-\d)B + \d e V) - \ord_E(K_{Y/X})
\le \val_E (B) - \ord_E(K_{Y/X})
\le 1 \le \val_E(V),
$$
and this is a chain of equalities if $E = E_i$ is the divisor dominating an
irreducible component $V_i$ that is extracted by $\m$.
On the other hand, if the center of $E$ intersects $V$ but is not contained in it,
then
$$
\val_E ((1-\d)B + \d e V) - \ord_E(K_{Y/X}) =
\val_E ((1-\d)B) - \ord_E(K_{Y/X}) < 1,
$$
since $(X,(1-\d)B)$ is Kawamata log terminal in a neighborhood
$U \subseteq X$ of $V$, and hence
$$
\lrd \val_E ((1-\d)B + \d e V) - \ord_E(K_{Y/X})\rrd \le 0 = \val_E(V).
$$
Since $V$ is a generically reduced pure-dimensional
scheme with no embedded components, its ideal sheaf has primary decomposition
$$
\I_V = \bigcap \I_{V_i},
$$
where $\I_{V_i}$ is the ideal
sheaf of the irreducible component $V_i$.
We have $\I_{V_i} = \m_*\O_{X'}(-E_i)$, and hence
$$
\J(X,(1-\d)B + \d e V)|_U = \I_V|_U.
$$
Therefore
$$
\J(X,(1-\d)B + \d e V) = \I_{V \cup W}
$$
for some subscheme $W$ in the complement of $U$, and hence disjoint from $V$.
This completes the proof of the lemma.
\end{proof}

Applying Nadel's vanishing theorem, we obtain
$$
H^i(X,\om_X\otimes \LL^{\otimes k}\otimes\AA\otimes\I_{V \cup W}) = 0
\for i > 0, \ k \ge d_1 + \dots + d_e.
$$
The deduce the stated vanishing from the following lemma.

\begin{lem}\label{lem:2}
Let $Z_1$ and $Z_2$ be two disjoint, nonempty, closed subschemes of $X$,
and suppose that for some nef and big line bundle $\M$ we have
$H^i(X,\om_X\otimes\M\otimes\I_{Z_1 \cup Z_2}) = 0$ for $i > 0$.
Then for each $j = 1,2$ we have
$$
H^i(X,\om_X\otimes\M\otimes\I_{Z_j}) = 0 \for i > 0.
$$
\end{lem}

\begin{proof}
Since $Z_1 \cap Z_2 = \emptyset$, we have the exact sequence
$$
0 \to \om_X\otimes\M\otimes\I_{Z_1 \cup Z_2} \to
\om_X\otimes\M \to
(\om_X\otimes\M|_{Z_1}) \oplus
(\om_X\otimes\M|_{Z_2}) \to 0.
$$
By Kawamata--Viehweg vanishing theorem, the higher cohomology groups of
$\om_X\otimes\M$ are trivial.
Therefore we see by the vanishing
$H^i(X,\om_X\otimes\M\otimes\I_{Z_1 \cup Z_2}) = 0$
for $i > 0$ that
$$
H^0(X,\om_X\otimes\M)
\to H^0(Z_1,\om_X\otimes\M|_{Z_1})
\oplus H^0(Z_2,\om_X\otimes\M|_{Z_2})
$$
is surjective and
$$
H^i(Z_1,\om_X\otimes\M|_{Z_1}) \oplus
H^i(Z_2,\om_X\otimes\M|_{Z_2}) = 0 \for i >0.
$$
These properties imply the surjectivity of the maps
$$
H^0(X,\om_X\otimes\M) \to H^0(Z_j,\om_X\otimes\M|_{Z_j})
$$
and the vanishings
$$
H^i(Z_j,\om_X\otimes\M|_{Z_j})) = 0 \for i > 0.
$$
We deduce from the exact sequences
$$
0 \to \om_X\otimes\M\otimes\I_{Z_j}
\to \om_X\otimes\M \to
\om_X\otimes\M|_{Z_j} \to 0
$$
that the cohomology groups $H^i(X,\om_X\otimes\M\otimes\I_{Z_j})$
vanish for $i > 0$.
\end{proof}

\begin{rmk}
By inversion of adjunction (see \cite[Theorem~8.1]{EM2}),
the theorem applies in particular to the case when $X$ is smooth and
$V$ is a $\Q$-Gorenstein subvariety with lci-defectively log canonical singularities.
Equivalently, one can use Corollary~\ref{thm:mld1} to adapt the above
proof to this setting.
\end{rmk}

\begin{rmk}\label{rmk:d_i-equal}
It is easy to see by inversion of adjunction that
the ideal sheaf $\I_V$ can be realized as the multiplier ideal
$\J(X,eV)$, and so one gets a very short proof of the theorem
in the special case when all the degrees $d_i$ are equal, as
the fact that $\LL^{\otimes d_1}\otimes\I_V$ is globally
generated implies immediately that
$$
H^i(\om_X\otimes\LL^{\otimes k}\otimes\AA\otimes\I_V) = 0 \for i > 0, \ k \ge ed_1
$$
by Nadel vanishing theorem.
However this realization of $\I_V$ as a multiplier ideal sheaf
does not yield the desired vanishing result
when the degrees $d_1,\dots,d_e$ are not all equal.
\end{rmk}

\section{Applications to regularity and projective normality}\label{sect:regularity}

Given a subscheme $V \subset \P^n$, we have introduced the notation
$$
d(V) := \min \{d_1 + \dots + d_e \mid
\text{$V$ is cut out by forms of degrees $d_1 \ge \dots \ge d_t$}\},
$$
where $e = \codim_{\P^n}(V)$.
As a particular case of Theorem~\ref{thm:vanishing}, we
obtain the following vanishing theorem for ideal
sheaves on projective spaces.

\begin{cor}\label{thm:vanishing-P^n}
Let $V \subset \P^n$ be a pure-dimensional proper subscheme with no embedded
components,
and assume that the pair $(\P^n,eV)$ is log canonical, where $e = \codim_{\P^n}V$.
Then
$$
H^i(\P^n,\I_V(k)) = 0 \for i > 0, \ k \ge d(V) - n.
$$
\end{cor}

The first three corollaries stated in the introduction follow from this result
by arguments analogous to those in \cite{BEL}, that we sketch below.

\begin{proof}[Proof of Corollary~\ref{thm:main-gen-case}]
For short, let $d = d(V)$.
The required vanishing
$$
H^i(\P^n,\I_V(d - e + 1 -i)) = 0 \for i >0
$$
follows for $1 \le i \le n-e+1$ by Corollary~\ref{thm:vanishing-P^n}
and for $i > n-e+1$ by dimensional reasons.

Regarding the second assertion, suppose that $V$
is non $(d - e)$-regular. Then necessarily
$$
H^{n-e}(\O_V(d - n-1))
= H^{n-e+1}(\P^n,\I_V(d - e -(n-e+1))) \ne 0.
$$
It follows by duality that
\begin{equation}\label{eq:nonvanishing}
H^0(\om_V(- d + n + 1)) = \Hom(\O_V(d-n-1),\om_V) \ne 0,
\end{equation}
where $\omega_V$ is the dualizing sheaf of $V$ (cf. \cite[Proposition~III.3.1]{Har}).

Let $d_1 \ge \dots \ge d_t$ be the degrees of the equations
cutting $V$ such that $d(V) = d_1 + \dots + d_e$, and
let $W$ be a general complete intersection of type
$(d_1,\dots,d_e)$ containing $V$. Note that $W$ is pure dimensional
with no embedded components.
Denoting by $I_V$ and $I_W$ the saturated ideals of $V$ and $W$
in the polynomial coordinate
ring of $\P^n$, let $R \subseteq \P^n$ be the subscheme defined by the ideal
$$
I_R = [I_W : I_V]
$$
(which is automatically saturated).
By \cite[Proposition~4.2.4]{Mig}, $R$ has no embedded points,
is equidimensional (of the same dimension of $W$), and is algebraically
linked to $V$ via $W$. In particular, $V$ is algebraically linked to $R$
via $W$, and thus we have an exact sequence of sheaves
$$
0 \to \I_W \to \I_R \to \omega_V(n+1-d) \to 0
$$
by \cite[Proposition~4.2.6]{Mig}.
Observing that $\I_{R/W} \cong \I_{R \cap V/V}$ since $W$ has no embedded
components, we obtain an isomorphism
$$
\om_V \cong \I_{R \cap V/V}(d - n-1).
$$
Going back to \eqref{eq:nonvanishing}, we conclude that
$H^0(\I_{R \cap V/V}) \ne 0$.
Since $W$ is connected, if $R$ were nonempty, then it would intersect
every connected component of $V$, and hence
$H^0(\I_{R \cap V/V})$ would be trivial. Therefore $R = \emptyset$.
\end{proof}

\begin{proof}[Proof of Corollary~\ref{thm:main-var-case}]
It follows from Corollary~\ref{thm:main-gen-case}, since
by inversion of adjunction $(\P^n,eV)$ is log canonical (cf. \cite{EM2,Kaw}).
\end{proof}

\begin{rmk}\label{rmk:CU2}
The bound on regularity obtained in \cite{BEL} also holds
allowing $V$ to be singular at a finite number
of points. Similarly, in the statement of \cite{CU}, $V$ is only assumed
to have rational singularities away from a one dimensional set
(the hypothesis of $V$ begin locally complete intersection
in \cite{CU} is also only needed away from a finite set).
At the moment it is unclear to us whether the methods applied in this paper
admit a similar extension.
\end{rmk}

\begin{proof}[Proof of Corollary~\ref{cor:proj-normal-CM}]
The proof is essentially the same as the one of \cite[Corollary~2]{BEL}.
Regarding the second part of the statement,
in order to get the vanishings
$$
H^i(\O_V(k)) = 0 \quad \text{for} \ 0 < i < \dim V \ \text{and} \ k > 0,
$$
we take a resolution of singularities $g \colon V' \to V$.
Since $V$ has rational singularities, we have
$H^i(V,\O_V(k)) = H^i(V',g^*\O_V(k))$, and thus the required vanishings
follow by Serre duality and Kawamata--Viehweg vanishing theorem.
\end{proof}

As in Corollary~\ref{thm:main-gen-case}, let
$V \subset \P^n$ be a pure-dimensional subscheme with no embedded components,
such that the pair $(\P^n,eV)$ is log canonical, where $e = \codim_{\P^n}V$.
Assume that $V$ is Cohen--Macaulay and smooth in codimension 1, and
let $d_1 \ge \dots \ge d_t$ be the degrees of the equations
cutting $V$ such that $d(V) = d_1 + \dots + d_e$.
We have seen that the condition $\reg(V) = d(V) - e +1$
characterizes $V$ being a complete intersection in $\P^n$, and it is natural
to ask whether the next case can be classified as well.
The arguments in the proof of Corollary~\ref{thm:main-gen-case}
yield the following characterization.

\begin{prop}\label{prop:CU1}
With the above notation, suppose that $V$ is not a complete intersection.
Then
$$
\reg(V) = d(V) - e
$$
if and only
either $H^0(\I_{R \cap V/V}(1)) \ne 0$ or $H^1(\I_{R \cap V/V}) \ne 0$,
where $R$ is the algebraic linkage of $V$ via a general
complete intersection of type $(d_1,\dots,d_e)$ containing $V$.
\end{prop}

\begin{rmk}\label{rmk:CU1}
The first of the two conditions considered in Proposition~\ref{prop:CU1}
is satisfied, for instance, if $V$ is a rational normal curve
of degree 3 in $\P^3$
(which in fact it is well known to have regularity $2$).
A generalization of this is treated in the next example.
In particular, there seems to
be a shift by $1$ in the definition of the notation $\reg(V)$
as adopted in \cite{CU}.
\end{rmk}

\begin{eg}\label{eg:residual}
For $1 \le k \le n-1$, fix a $k$-dimensional linear subspace $\L \subset \P^n$,
and let $V \subset \P^n$ be the residual intersection of
$e := n-k$ general quadrics containing $\L$.
We claim that $\reg(V) = d(V) - e$.

First note that $V$ is Cohen--Macaulay (see, e.g., \cite[Corollary~4.2.9]{Mig}),
and it is not difficult to check that it is smooth in codimension one,
so in particular $V$ is normal.
Notice also that $V$ is not complete intersection, as it is a non-degenerate variety
of degree $2^e - 1$ and codimension $e$.

We observe that $(\P^n,eV)$ is log canonical. To see this,
we construct a log resolution $f \colon Y \to \P^n$ as follows.
We first take the blow up $g \colon Y' \to \P^n$ along $\L$.
Let $E$ be the exceptional divisor of $g$.
For a general choice of the $e$ quadrics vanishing along $\L$,
the proper transform $V' \subset Y'$ of $V$ is smooth and intersects
transversally $E$.
Therefore the blow up of $Y'$ along $V'$
produces the required log resolution of $(\P^n,eV)$, and it is immediate to
check that the discrepancies of the two exceptional divisors
along this pair are $\ge -1$, so that the pair is log canonical.

By construction, we have $\O_V(-1) \subseteq \I_{\L \cap V/V}$, and
thus
$$
h^0(\I_{\L \cap V/V}(1)) \ge h^0(\O_V) > 0.
$$
Therefore we conclude that $\reg(V) = d(V) - e$ by Proposition~\ref{prop:CU1}.
\end{eg}

We conclude the paper with a discussion of the
following classical constructions giving rise to settings
that satisfy the assumptions in Corollary~\ref{thm:main-gen-case}.

\begin{eg}\label{eg:determinant}
For $s \ge r \ge 1$, let $V \subset \P^{rs - 1}$ be defined
by the $r$-minors of an $r \times s$ matrix with general linear entries.
We can choose coordinates $x_{i,j}$, where $1\le i \le r$ and $1 \le j \le s$,
so that $\P^{rs-1} = \Proj \C[x_{i,j}]$ and $M = [x_{i,j}]$.
For every $k = 1, \dots, r-1$, let $V_k \subset \P^{rs-1}$ be the locus
where $\rk M \le k$. Then
$$
\mult_{V_k}(V) = r-k \and
\codim_{\P^{rs-1}}(V_k) = (s-k)(r-k)
$$
for every $k$, so that, in particular,
$V = V_{r-1}$ and $\Sing(V) = V_{r-2}$.

We consider the sequence of blowups $g_k\colon Y_k \to Y_{k-1}$, where
$Y_0 = \P^{rs-1}$ and $g_k$ is the blow up of $Y_{k-1}$
along the proper transform of $V_k$. An explicit computation shows
that each center of
blow up is smooth and that this operation produces a log resolution
$$
f \colon Y = Y_{r-1} \to \P^{rs-1}
$$
of $(\P^{rs-1},V)$.
Let $E_1,\dots,E_{r-1}$ be the exceptional divisors of $f$, labeled so that
$f(E_k) = V_k$. Then
$$
\ord_{E_k}(K_{Y/\P^{rs-1}}) - e \val_{E_k}(V) = (s-k)(r-k) - 1 - (s-r+1)(r-k) \ge -1
$$
for all $k$, which implies that $(\P^{rs-1},eV)$ is log canonical.
Therefore $V$ satisfies the assumptions in Corollary~\ref{thm:main-gen-case},
and hence, in particular, we recover
the well known fact that $\reg(V) \le d(V) - e + 1$.
\end{eg}

\begin{eg}\label{eg:reg-generic-proj}
Let $V$ is the image of a smooth $e$-dimensional variety $\~V$ in $\P^{2e+1}$
under a general projection.
Note that $V$ has at most a finite number of double
points and is smooth elsewhere. Moreover, if $p$ is a double point
of $V$, then the tangent cone of $V$ at $p$ (viewed as a subset
of the tangent space $T_p\P^{2e}$) is the union of two
linear spaces meeting only at the origin (and hence spanning $T_p\P^{2e}$).

This can be checked as follows. Consider the incidence sets
\begin{gather*}
I = \{(p,q,o) \in \~V\times\~V\times\P^{2e+1} \mid p \ne q, \,
o \in \ov{pq} \}, \\
I' = \{(p,q,o) \in I \mid
\dim (\Span(PT_p\~V,o) \cap \Span(PT_q\~V,o)) \ge 2 \},
\end{gather*}
where $PT_p\^V \subset \P^{2e+1}$ denotes the linear subspace of dimension $e$
tangent to $\~V$ at $p$. Note that $I$ is irreducible
and the projection to the last factor $I \to \P^{2e+1}$
is generically finite whenever dominant. On the other hand, $I'$ is a proper
closed subset of $I$. Indeed, if $(p,q)$ is general in $\~V \times \~V$,
then the linear projection $\pi_p \colon \~V \setminus \{p\} \to \P^{2e}$
has maximal rank at $q$, which means that
$PT_p\~V \cap PT_q\~V = \emptyset$, and hence
$\Span(PT_p\~V,o) \cap \Span(PT_q\~V,o) = \ov{pq}$ for every $o \in \ov{pq}$.
This implies that the projection to the last factor $I' \to \P^{2e+1}$
is not dominant.

Suppose that $p$ is a double point of $V$.
Let $h\colon Y \to \P^{2e}$ be the blow up of $\P^{2e}$ at $p$,
and let $E$ be the exceptional divisor. If $V' \subset Y$ is the proper
transform of $V$, then $V' \cap E$ is the union of two disjoint planes,
and thus $V'$ is smooth and intersects transversally $E$.
Therefore, in order to prove that $(\P^{2e},eV)$ is log canonical,
it suffices to check the discrepancy along $E$. We have
$$
\ord_E(K_{Y/\P^{2e}}) - e \val_E(V) = 2e-1 - 2e = -1,
$$
and thus we can conclude that the pair is log canonical, and hence
$V$ satisfies the assumptions in Corollary~\ref{thm:main-gen-case}, and
therefore
$$
\reg(V) \le d(V) - e + 1.
$$
Note that the tangent cone of $V$ at a double point $p$ spans the whole
tangent space of $\P^{2e}$. This implies in particular that
if $e \ge 2$ and $V$ is singular, then $V$ is not locally complete intersection,
and hence $\reg(V) \le d(V) - e$.
\end{eg}

\begin{rmk}
Still keeping the notation as in Example~\ref{eg:reg-generic-proj},
it is interesting to observe that
if $e = 2$ and we (generically) project further to a surface $\ov{V} \subset \P^3$, then
$(\P^3,\ov{V})$ is also log canonical, so it still
satisfies the hypothesis of Corollary~\ref{thm:main-gen-case}. Of course
this is irrelevant from the point of view of regularity, as in this
case $\ov{V}$ is a divisor.
By contrast, if $e \ge 3$, then taking a general
projection to an hypersurface $\ov{V} \subset \P^{e+1}$ produces a pair
$(\P^3,\ov{V})$ that in general is not log canonical, as the multiplicities
may grow exponentially with respect to $e$.
One can still ask whether there is a function $\phi(e)$,
with $e < \phi(e) < 2e$ such that a general projection
to $\P^{\phi(e)}$ gives a variety to which Corollary~\ref{thm:main-gen-case} applies.
\end{rmk}

\providecommand{\bysame}{\leavevmode \hbox \o3em
{\hrulefill}\thinspace}

\end{document}